\documentclass{amsart}
\usepackage{graphicx} 
\newtheorem{theorem}{Theorem}
\newtheorem{conjecture}[theorem]{Conjecture}

\newtheorem{corollary}[theorem]{Corollary}

\newtheorem{lemma}[theorem]{Lemma}
\newtheorem{definition}[theorem]{Definition}
\usepackage{hyperref,xcolor}
\usepackage[T1]{fontenc}
\usepackage{pdfsync}
\DeclareMathSymbol{\lsim}{\mathord}{symbols}{"18}
\usepackage{tikz}
\usetikzlibrary{arrows.meta}

\title{On the records and zeros of a deterministic random walk}
\author{Henk Bruin \and Robbert Fokkink}
\email{henk.bruin@univie.ac.at \and r.j.fokkink@tudelft.nl}
\date{\today}

\begin{document}

\begin{abstract}
We settle two questions on sequence A120243 in the OEIS that were raised by Clark Kimberling 
and partly solve a conjecture of Van de Lune and Arias de Reyna. 
We extend Kimberling's questions to the
framework of deterministic random walks, automatic sequences, and linear recurrences. 
Our results indicate that there
may be a deeper connection between these structures. In particular, we conjecture that
the records of deterministic random walks
are $\xi$-Ostrowski automatic for a quadratic rotation number $\xi$. 
\end{abstract}
\keywords{deterministic random walk, record, zero, Pell automatic sequence} 
\subjclass[2020]{11B85, 11K38, 60J15} 

\maketitle

Sequence \href{https://oeis.org/A120243}{A120243} in the OEIS consists of numbers $n$ such that the fractional part $\{n \sqrt{2}\}$ is less than $\frac{1}{2}$. Its complementary sequence, denoted as $b(n)$, is \href{https://oeis.org/A120749}{A120749}. The table below presents the first eighteen elements of these sequences:

\begin{table}[h!]
\centering
\small
\begin{tabular}{|c|cccccccccccccccccc|}
\hline
$a(n)$ & 1 & 3 & 5 & 6 & 8 & 10 & 13 & 15 & 17 & 18 & 20 & 22 & 25 & 27 & 29 & 30 & 32 & 34 \\ \hline
$b(n)$& 2 & 4 & 7 & 9 & 11 & 12 & 14 & 16 & 19 & 21 & 23 & 24 & 26 & 28 & 31 & 33 & 36 & 38 \\ \hline
\end{tabular}
\\[2mm]
\caption{The first eighteen numbers from A120243 and A120749}\label{tbl:1}
\end{table}
\vspace{-0.4cm}

These sequences were entered into the OEIS by Clark Kimberling, who posed the question of whether the difference $b(n) - a(n)$ is positive for all $n$ and whether, for each integer $k$, there exist infinitely many values of $n$ such that $b(n) - a(n) = k$. We confirm both of these properties.
Kimberling's questions relate to the rotation of the unit circle by $\sqrt{2}$, which is often an initial case leading to broader mathematical results, such as in \cite{beck2010}. Our analysis of the sequences $a(n)$ and $b(n)$ extends to the more general framework of automatic sequences and ergodic theory.

We could also have described the sequences A120243 and A120749 by the parity of $\lfloor 2n\sqrt 2\rfloor$.
The sequence $a$ corresponds to the even numbers, and $b$ corresponds to the odd numbers.
Or alternatively, if instead of parity we use signs by putting
$(-1)^{\lfloor 2n\sqrt 2\rfloor}$, then we can think of the
sequences $a$ and $b$ as the steps in the
positive direction and the negative direction, respectively, in what is known as, rather curiously, a
\emph{deterministic random walk}
\begin{equation}\label{eq:1}
  S_n(\xi)=\sum_{j=1}^n (-1)^{\lfloor j\xi\rfloor},
\end{equation}
for $\xi=2\sqrt 2$. 
If we move from fractional parts $\{n\xi\}$ to deterministic random walks $S_n(2\xi)$,
then we double $\xi$.
Kimberling's question whether $b(n)-a(n)$ is positive is
equivalent to the question whether $S_n(2\sqrt 2)$ is non-negative.
This problem closely resembles Problem B6 from the 81st William Lowell Putnam Mathematical Competition, which requires proving that \( S_n(\sqrt{2} - 1) \) is non-negative, see \cite{bhargava}.

A number $n$ is a \emph{record} of the deterministic random walk
if none of the previous
partial sums (including zero) is equal to $S_n(\xi)$. 
Sequence
\href{https://oeis.org/A123737}{A123737} in the OEIS
contains the partial sums of the deterministic random walk 
for $\xi=\sqrt 2$.
In a comment on A123737, V\'{a}clav 
Kot\v{e}\v{s}ovec asked if the 
records satisfy a certain recurrence relation.
Van de Lune and Arias de Reyna~\cite{lune2008}
conjecture that there is a system of recurrence relations for
the records of the deterministic random walk $S_n(\xi)$
for all
quadratic irrationals $\xi$.
Their work extends earlier results of O'Bryant et al~\cite{obryant2006}.
We confirm Kot\v{e}\v{s}ovec's recurrence and 
settle some instances of the conjecture.

A number $n$ is a \emph{zero} of the deterministic random walk
if $S_n(\xi)=0$.
Deterministic random walks
$S_n(\xi)=\sum_{j=1}^n (-1)^{\lfloor x + j\xi\rfloor}$ (with offset $x\in[0,1)$)
were introduced by Aaronson and Keane~\cite{aaronson1982visitors},
and their primary interest was an estimate of the asymptotic number of zeros
as $n$ goes to infinity for a generic number~$x$. 
More specifically, if $N_n$ is the number of zeros up to $n$,
then it can be interpreted as a random variable depending on a 
uniformly random $x$. The asymptotic mean and variance of $N_n$
remain a subject of ongoing research~\cite{avila2015visits, bruin2024}.

\section{Irrational rotations and automata}

The sequences $a$ and $b$ arise from an irrational rotation of
the unit circle. The number $a(n)$ marks the $n$-th return to the
semi-circle $[0,\frac 12)$ under the rotation $x\mapsto x+\sqrt 2$,
starting from $x=0$,
and $b(n)$ is the $n$-th return to its complement $[\frac 12,1)$.
According to the ergodic theorem, both $\lim\frac {a(n)}n$
and $\lim\frac {b(n)}n$ equal two, which is suggested already by the first
eighteen entries in
Table~\ref{tbl:1}. 
Given this, the next point of interest is 
the behavior of the differences $a(n)-2n$ and $b(n)-2n$.
Such deviations from the mean are known as \emph{discrepancies}, which is a significant topic in the study of sequences ~\cite{drmota2006}. Kimberling’s question regarding whether 
$b-a$ assumes every positive number infinitely often falls within this topic.
His other question on the signature of $b-a$ is equivalent to asking whether
$a(n)-2n$ is negative and $b(n)-2n$ is non-negative.

For $N\in\mathbb N$ an interval $I\subset (0,1)$ and $\xi\in (0,1)$, 
let $\ell(I)$ be the length of $I$ and let $N(\xi,I)$ be the cardinality of
$\{n\colon \{n\xi\}\in I,\ n\leq N\}$. Kesten~\cite{kesten1966} famously proved that
\begin{equation}\label{discrep}
    \limsup_{N\to\infty} |N(\xi,I)-\ell(I)\cdot N|<\infty
\end{equation}
if and only if $\ell(I)=\{m\xi\}$ for some $m\in\mathbb N$.    
If we apply this to $\xi=\{\sqrt 2\}$ and $I=[0,\frac 12)$, then we get that
the difference between $N/2$ and the number of $\{n\xi\}<\frac 12$ up to $N$
is unbounded. This implies that $a(n)-2n$ and $b(n)-2n$ are both unbounded.

S\'os~\cite{sos1957} observed that the discrepancy 
$N(\xi,I)-\ell(I)\cdot N$ may be bounded on one side and unbounded on the other.
Recall the (regular) continued fraction of a real number $\xi$
\[
\xi = a_0+\cfrac{1}{a_1+\cfrac{1}{a_2+\cfrac{1}{a_3+\cfrac{1}{\ddots}}}}
\]
The coefficients $a_i$ are the partial quotients and the finite expansions
$\frac{p_n}{q_n}$ are the convergents, starting from $\frac{p_0}{q_0}=\frac {a_0}1$.
Dupain and S\'{o}s~\cite{dupain1980} gave a necessary and sufficient condition on $I=[0,\beta)$ and $\xi$ for one-sided
boundedness, for $\xi$ with bounded partial quotients (such as quadratic irrationals).
Boshernitzan and Ralston~\cite{boshernitzan2009} found an elegant condition for 
nonnegative discrepancy in terms of the convergents
of $\xi$.
\begin{theorem}[Boshernitzan and Ralston]\label{thm:BR}
    Let $I=[0,\frac hk)$ and let $\frac {p_n}{q_n}$ be the convergents of $\xi$.
    The discrepancy $N(\xi,I)-\ell(I)\cdot N$ is nonnegative if and only if $k\ |\ q_{2n+1}$ for all $n$.
\end{theorem}
In particular, if the odd convergents $q_{2n+1}$ of $\xi$ are even then the number
of $\{n\xi\}$ in $[0,\frac 12)$ up to $N$ exceeds, or is equal to, the number
of $\{n\xi\}$ in $[1/2,1)$. This is equivalent to the non-negativity of $S_n(2\xi)$.
The denominators $q_n$ of the convergents of $\sqrt 2$, starting from $q_0=1$, 
are the Pell numbers $1, 2, 5, 12, \ldots$,
entry \href{https://oeis.org/A000129}{A000129} in the OEIS. The $q_n$ are even for odd $n$ and so the discrepancy is nonnegative by Theorem~\ref{thm:BR}. For each $N$ at least half of the iterates $\{n\sqrt 2\}$ are in $[0,\frac 12)$.
It follows that $a(n)-2n$ is negative and $b(n)-2n$ is nonnegative for all $n$, and therefore $b-a$ is positive.
This settles one of the two questions of Kimberling.
\begin{corollary}
    The sequence $b-a$ is positive.
\end{corollary}
The odd convergents are even if and only if the odd partial quotients $a_{2n+1}$ are even.
Problem B6 of the 2021 Putnam Competition asks one to prove that $S_n(\sqrt 2 -1)$ is non-negative.
The continued
fraction expansion of $\frac{\sqrt 2 -1 }{2}$ is $[0;\overline{4,1}]$ where the bar indicates that
these coefficients are repeated. The partial quotients $a_{2n+1}$ are all equal to $4$ so that
Theorem~\ref{thm:BR}
settles problem B6.

The Pell numbers $P_n$ form the basis of a numeration system~\cite[Ch 3.4]{shallit2022}.
They satisfy the recurrence $P_{n+1}=2P_n+P_{n-1}$ starting from $P_0=0,\ P_1=1$.
Please note that there is a mismatch between the indexing of the Pell numbers and the denominators
of the convergents of $\sqrt 2$. The Pell numbers start at $P_0=0$ and the
denominators start at $q_0=1$.
Each number can be represented
as a sum $n=\sum_{i=1}^j d_iP_i$ with digits $d_i\in\{0,1,2\}$. 
The representation is unique under the condition that $d_i=0$ if $d_{i+1}=2$.
For example, $69=1\cdot 1+ 0\cdot 2+ 2\cdot 5 + 0\cdot 12 + 2\cdot 29$.
As in decimal notation, it is standard practice to write the digits
with the most significant digit first (msd). 
The decimal number $69$ is represented by $20201$ in Pell numeration.
A sequence is \emph{Pell automatic} if there exists a deterministic finite-state automaton (DFA)
that reads digits in Pell numeration and decides if a number is in the sequence.
The DFA in Fig.~\ref{fig:1} decides if a number is in \href{https://oeis.org/A120749}{A120749} or not. For example,
$69$ is entered as $20201$, leading to the state transitions
$0\to 2\to 4\to 2\to 4 \to 5$ ending in the accepting state $5$. The number is in the sequence.
In fact, inspection of the OEIS shows that it is the 34th element of the sequence, which implies
that the 69th step of the deterministic random walk has value 1. The next step, seventy,
is Pell number $P_6$, which has representation $100000$. It ends in state 3, which is accepting.
Therefore, the walk has value 0 at this point: $P_6$ is a zero of $S_n(2\sqrt 2)$. We will prove
below that the zeros and records of this walk are Pell automatic.
\begin{figure}[b]\label{fig:1}
    \centering
    \includegraphics[width=0.9\linewidth]{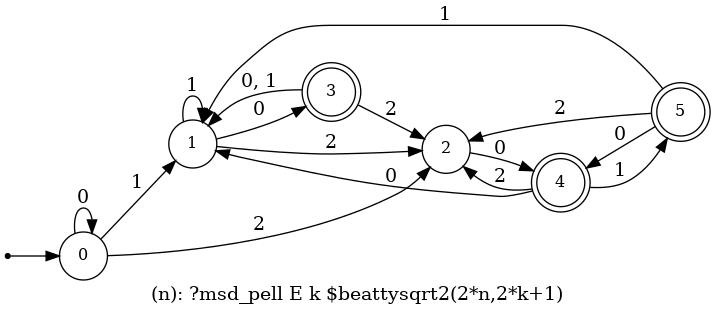}
    \caption{A finite state automaton that decides if a number $n$ is in A120749 or in its complement A120243. The input is in msd format. Inputs that end in double
    circled states (acceptance) are in A120749. Inputs that end in state 1 (rejection) are in
    A120243. Inputs that end in $2$ are not a valid Pell representation. The
    automaton is produced by the automatic theorem prover
    \texttt{Walnut} using the results in~\cite{schaeffer2024}.}
    \label{DFA120749}
\end{figure}

The DFA in Fig.~\ref{fig:1} is constructed from a method due
to Schaeffer et al.~\cite{schaeffer2024},
who showed how a Beatty sequence such as $\lfloor n\sqrt 2\rfloor$ can
be implemented in the automatic theorem prover~\texttt{Walnut}~\cite{mousavi2016}. 
In particular, it is possible to implement a command
\texttt{beattysqrt2(n,r)} which accepts numbers $r=\lfloor n\sqrt 2\rfloor$.
Sequence \href{https://oeis.org/A120749}{A120749} contains
the numbers $n$ such that $\lfloor 2n\sqrt 2\rfloor$ is even. In terms of first-order logic, a
number $n$ is in the sequence if
\[\exists k\ \lfloor 2n\sqrt 2\rfloor=2k+1.\]
In the \texttt{Walnut} environment this becomes
\begin{center}
    \texttt{"?msd\_pell E k \$beattysqrt2(2*n,2*k+1)":}
\end{center}
which produces the depicted DFA.

\section{The records of $S_n(2\sqrt 2)$}

We say that $r$ is a \emph{record} if $S_r(\xi)=m$ and none of the $S_n(\xi)$ for $n<r$ are equal to $m$.
The sequence $R_n$ of records of $S_n(\sqrt 2)$ starts as
\[
0,\ 1,\ 3,\ 8,\ 20,\ \ldots
\]
Jan van de Lune~\cite{lune1984} conjectured 
that the records of $S_n(\sqrt 2)$ 
satisfy a Pell-like recurrence, which was confirmed in~\cite{fokkink1994}.

\begin{theorem}[Van de Lune]\label{thm:lune}
The sequence of records $R_n$ of $S_n(\sqrt 2)$ satisfies the recurrence
\begin{equation}\label{eq:Lune}
R_{n+1}=2R_n+R_{n-1}+1,\ \text{with } R_0 = 0, R_1 = 1.
\end{equation}
and the values of consecutive records have alternating signs.
\end{theorem}

The records are sums of Pell numbers $R_n=\sum_{j=1}^n P_j$. In Pell numeration,
their representation is $11\cdots 1$ and the sequence $R_n$ is Pell automatic.
\begin{figure}[h]
    \centering
    \includegraphics[width=0.4\linewidth]{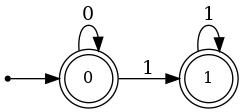}
    \caption{An automaton that accepts the records of $S_n(\sqrt 2)$ in
    Pell numeration in msd representation.}
    \label{fig:rec}
\end{figure}

V\'{a}clav 
Kot\v{e}\v{s}ovec separated the records into positive and negative
values. For a natural number $m$, let $A_m$ be the first index such 
that $S_n(\sqrt 2)=m$ and let $B_m$ be the first index such that
$S_n(\sqrt 2)=-m$, both starting out from $A_0=B_0=0$. Kot\v{e}\v{s}ovec 
noticed that these are recurrence
sequences and that $A_n$ occurs as \href{https://oeis.org/A001652}{A001652}
while $B_n$ occurs as \href{https://oeis.org/A001108}{A001108} in the OEIS.

\begin{corollary}
The sequences $A_n$ and $B_n$ satisfy the recurrences    
\[
A_{n+1} = 6A_{n} - A_{n-1} + 2,\ \text{with } A_0 = 0, A_1 = 3,
\]
and
\[
B_{n+1} = 6B_{n} - B_{n-1} + 2,\ \text{with } B_0 = 0, B_1 = 1.
\]
\end{corollary}
\begin{proof}
This
follows from Van de Lune's recursion \eqref{eq:Lune} and the observation that the first
step is in the negative direction:
\[
\begin{aligned}
R_{n+2} &= 2\big(2R_n + R_{n-1} + 1\big) + R_n + 1 \\
       &= 5R_n + 2R_{n-1} + 3 \\
       &= 6R_n - R_{n-2} + 2.
\end{aligned}
\]
\end{proof}

The first few records of $S_n(2\sqrt 2)$ are 
\[
0,\ 1,\ 6,\ 35,\ 204,\ \ldots
\]
which happen to be the first few half Pell numbers $P_{2n}/2$. 
These numbers satisfy the recurrence relation $Q_{n+1}=6Q_n-Q_{n-1}$,
which is Kot\v{e}\v{s}ovec's recursion up
to a constant.
To verify that the records $Q_n$ are indeed half Pell numbers, we apply
the algorithm from \cite{fokkink1994} to compute values of the deterministic random
walk $S_n(\xi)$. 
It assumes that $0<\xi<1$ and
depends on the denominators of the convergents $q_n$
of $\xi/2$. In our case $\xi=2\sqrt 2$, we translate
to $2\sqrt 2-2$ to place it in the unit interval and
divide by two, to get $\sqrt 2-1$.
The denominators of its convergents are the Pell numbers
(starting from $q_0=1=P_1$). 
For this particular case we
have the following three rules. We write $q=q_{n+1}$, $q'=q_{n}$, 
and $q''=q_{n-1}$ and we write $S_n$ for $S_n(2\sqrt 2)$. The rules are,
in this special case of $\xi=2\sqrt 2$:

\begin{center}
    \fbox{%
        \begin{minipage}{0.75\textwidth}  
        \hfil Recursive rules for $S_n(2\sqrt 2)$\hfil
            \begin{enumerate}
                \item[] Rule A: $S_q = 1$ if $q$ is odd and $S_q = 0$ if $q$ is even.
                \item[] Rule B: $S_{q-k} = S_{q'} + S_{k-1}$ if $1 \leq k \leq q''$.
                \item[] Rule C: $S_{q'+k} = S_{q'} + S_k$ if $1 \leq k < q'$.
            \end{enumerate}
        \end{minipage}
    }
\end{center}
An explanation for these rules is that $q\xi$ is a close return to 0 mod 1. If $q$ is
even, there are equally many steps in the two semicircles $[0,1/2)$ and $[1/2,1)$.
If $q$ is odd, then it is the denominator of a convergent $p/q<\xi$. The return
$q\xi$ is in $[0,1/2)$ and therefore there is one more step in $[0,1/2)$.

Rule B determines $S_n$ for $n\in [q-q'',q)$ and Rule C determines $S_n$ for $n\in (q',2q')$.
Since $q-q''=2q'$ these rules, together with Rule A, determine all values $S_n$ by recursion.
The parity of the Pell numbers alternates. If $q$ is odd than $q'$ is even and vice versa.
If the interval $[q',q]$ is marked by an even $q'$ and an odd $q$, 
then the contributions $S_q'$ in rule B and rule C are zero.
It follows that there is no record in $[q',q]$ in this case. Records
can only occur if $q'$ is odd and $q$ is even. 
Notice that a Pell number $P_n$
is even if and only if its index is even.

\begin{theorem}\label{thm:rec2}
    The records $Q_n$ of $S_n(2\sqrt 2)$ are the half-Pell numbers $P_{2n}/2$.
\end{theorem}
Note that the half-Pell number $P_{2n}/2$ is the sum of the first
$n$ odd Pell numbers. By rules A and C, each odd Pell number contributes
1 to the walk at step $P_{2n}/2$.
\begin{proof}
We can limit our attention to intervals $[q',q]$ such that $q'$ is odd.
These are the intervals $[P_{2n-1},P_{2n}]$.
For $n=1$ we
have $P_2=2$ and the unique record in $[1,2]$ occurs indeed at $1=P_2/2$.
By rules B and C the value of $S_n$ increases by $S_{q'}=1$ for $n\in (q',q)$ compared to earlier values,
so there is at most one record in this interval. By rule C $S_{q'+k}(\sqrt 2)=1+S_k(\sqrt 2)$
for all $k\in [0,q')$ and therefore the record has to come from this rule.
The new record in $[q',q]$ is $q'+r$
for the old record $r\in [0,q')$. By induction, records occur as 
sums of odd Pell numbers $P_1+P_3+\cdots+P_{2m-1}=P_{2m}/2$.
\end{proof}

\begin{figure}[h]
    \centering
    \includegraphics[width=0.6\linewidth]{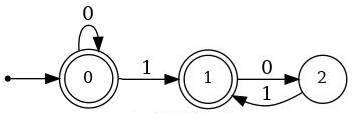}
    \caption{An automaton that accepts the records of $S_n(2\sqrt 2)$ in
    Pell numeration in msd representation.}
    \label{fig:rec2}
\end{figure}
If we ignore the first record at 0, then
in Pell numeration the records of $S_n(2\sqrt 2)$ occur at $1,\ 101,\ 10101,\ \ldots$ and more generally $(10)^*1$, where the Kleene $*$ represents an arbitrary repetition. Again, they form a Pell automatic sequence.
The Ostrowski numeration system can be defined for every real number $\xi$, see \cite[ch 3.5]{shallit2022}.
It is equal to Pell numeration if $\xi=\sqrt 2$. 
\begin{definition} \quad Given a positive real irrational number $\xi = [a_0, a_1, \dots]$  
with continued fraction convergents $p_n/q_n = [a_0, a_1, \dots, a_n]$, we can write every integer $N \geq 0$ uniquely as  
\[
N = \sum_{0 \leq i \leq j} b_i q_i
\]
where the digits $(b_i)_{i \geq 0}$ satisfy the conditions

\begin{itemize}
    \item[(a)] $0 \leq b_0 < a_1$.
    \item[(b)] $0 \leq b_i \leq a_{i+1}, \quad \text{for } i \geq 1$.
    \item[(c)] \textit{For } $i \geq 1$, \textit{if } $b_i = a_{i+1}$, \textit{then } $b_{i-1} = 0$.
\end{itemize}
\end{definition}

A sequence is $\xi$-Ostrowski automatic if there exists a finite state automaton that decides
if a number is in the sequence.
The following conjecture is a variation of
the conjecture of van de Lune and Arias de Reyna, which we mentioned earlier.

\begin{conjecture}\label{conj:1}
    The records of $S_n(2\xi)$ are $\xi$-Ostrowski automatic for quadratic irrational $\xi$.
\end{conjecture}

\begin{figure}[h]
    \centering
    \includegraphics[width=0.9\linewidth]{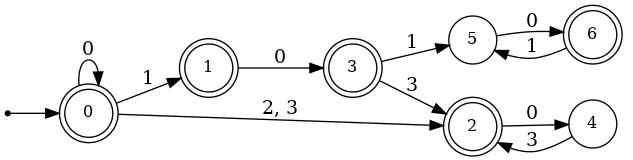}
    \caption{An automaton that accepts the records of $S_n(\sqrt 3)$ in
    $\sqrt 3/2$-Ostrowski numeration, provided that the system of recurrences
    conjectured by Van de Lune and Arias de Reyna holds. The continued
fraction of $\sqrt{3}/2$ has partial coefficients $[0;1,\overline{6,2}]$, where the bar marks that
these coefficients repeat. The denominators of its convergents are $1, 7, 15, 97, \ldots$.}
    \label{fig:recsqrt3}
\end{figure}
To back up their conjecture,
Van de Lune and Arias de Reyna provide a specific recurrence relation for the records of $\xi=\sqrt 3$,
with initial values $t_{j}=0$ for negative indices (where we corrected a typo for $t_{4n+1}$).
\[
\begin{aligned}
    t_{4n+1} &=2 t_{4n} + t_{4n-1} + 1 \\
    t_{4n+2} &= t_{4n+1} + 2 t_{4n} + 1 \\
    t_{4n+3} &= t_{4n+2} + 2 t_{4n} + 1\\  
    t_{4n+4} &= 2 t_{4n+3} + t_{4n} + 1. 
\end{aligned}
\]
This recurrence was found experimentally. It generates the sequence
\[
1,\ 2,\  3,\  7,\  18,\   33,\   48,\   104,\   257,\   466,\ 675,\ 1455,\ 3586, \ \ldots 
\]
which unfortunately does not yet match any sequence in the OEIS. It is possible to construct an automaton,
shown in Fig.~\ref{fig:recsqrt3},
for this system of recurrences using the automatic theorem prover \texttt{Walnut}. 

\section{The zeros of $S_n(2\sqrt 2)$}

An index $n$ is a zero of a deterministic random walk if $S_n(\xi)=0$.
The study of the asymptotic number of zeros of deterministic random walks with offsets
was initiated in~\cite{aaronson1982visitors} and remains a topic of
ongoing research.
For our walk $S_n(2\sqrt 2)$ the first few zeros are
\[
0,\ 2,\ 4,\ 12,\ 14,\ 16,\ 24,\ 26,\ 28,\ 70,\ 72, \ 74,\ 82,\ 84,\ 86,\ \ldots
\]
which is entry \href{https://oeis.org/A194368}{A194368} of the OEIS. In Pell numeration
these numbers are
\[
\epsilon ,\ 10,\ 20,\ 1000,\ 1010,\ 1020,\ 2000,\ 2010,\ 2020,\ 100000,\ \ldots
\]
where $\epsilon$ is the empty word. 
Like its records, the zeros of $S_n(\sqrt 2)$ are easy to spot in Pell numeration. 
\begin{figure}
    \centering
    \includegraphics[width=0.61\linewidth]{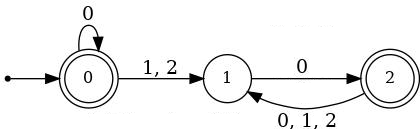}
    \caption{An automaton that accepts the zeros of $S_n(2\sqrt 2)$ in
    Pell numeration in msd representation.}
    \label{fig:zeros}
\end{figure}
\begin{theorem}
    The zeros of $S_n(2\sqrt 2)$ occur at $n=0$ or $(10|20)(00|10|20)^*$ in Pell numeration.
\end{theorem}
\begin{proof}
    There cannot be a zero in $[q',q]$ if $q'$ is an odd Pell number because rules
    $B$ and $C$ add $S_{q'}(\sqrt 2)=1$ to previous values, and these are non-negative.  
    Zeros occur in $[q',q]$ for even Pell numbers $q'$, starting at $q'$
    following rule A.
    Note that rules B and C partition this interval
    into numbers with initial digits 1 and 2 in Pell numeration.

    We assume by induction that the earlier
    zeros all are of the required form.     
    If rule C applies, then $S_{q'+k}(\sqrt 2)=0$
    if and only if $S_k(\sqrt 2)=0$, so indeed all zeros in this part
    of $[q',q]$ are of the form $(10)(00|10|20)^*$.
    For rule B we use the identity
    \[
    P_{2m}-1=\sum_{i=1}^{m-1} 2P_{2i}.
    \]    
    We rewrite rule B as $S_{(q-1)-k)}(\sqrt 2)=S_{k}(\sqrt 2)$, using that $q'$
    is an even Pell number. In particular, $q-1-k$ is a zero if and only if $k$
    is a zero. Now $q-1$ is equal to $(20)^m$ in Pell numeration and $k=(00|10|20)^{m-1}$
    by our inductive assumption. It follows that $(q-1)-k=(20)(00|10|20)^{m-1}$
    and we are done.
    \end{proof}

Surprisingly, the set of zeros of $S_n(\sqrt 2)$ does not seem to be Pell automatic,
or if it is, its automaton needs many states.
The one-sidedness of $S_n(2\sqrt 2)$ seems to be essential.
We say that $\xi$ is a \emph{BR-number} (after Boshernitzan and Ralston) if its odd convergents $q_{2n+1}$
are even, or, equivalently, if all its odd partial quotients $a_{2n+1}$ are even.
According to Theorem~\ref{thm:BR}, the walk $S_n(2\xi)$ is nonnegativeif and
only if $\xi$ is a BR-number. Note that $q_0=1$ and $q_1$ is even for a BR-number.
By the recursion $q_{n+1}=a_nq_n+q_{n-1}$, the parity of the convergents
alternates for BR-numbers.

We write $q=q_{n+1}$, $q'=q_{n}$, 
and $q''=q_{n-1}$ so that $q=aq'+q''$ for the partial quotient $a=a_{n+1}$.
The rules B and C from~\cite{fokkink1994} allow some overlap, 
but we state them in such a form that they apply to separate parts of $[q',q)$. 
\begin{center}
    \fbox{%
        \begin{minipage}{0.75\textwidth}  
        \hfil Recursive rules for $S_n(2\xi)$ for a BR-number $\xi\in (0,1)$ \hfil
            \begin{enumerate}
                \item[] Rule A: $S_q = 1$ if $q$ is odd and $S_q = 0$ if $q$ is even.
                \item[] Rule B: $S_{q-k} = S_{q'} + S_{k-1}$ if $1 \leq k \leq q/2$.
                \item[] Rule C: $S_{q'+k} = S_{q'} + S_k$ if $1 \leq k < q/2-q'$.
            \end{enumerate}
        \end{minipage}
    }
\end{center}

\begin{theorem}\label{thm:zeroBR}
    Let $\xi\in (0,1)$ be a BR-number and let $N=\sum_{i=0}^n b_iq_i$ be the $\xi$-Ostrowski
    representation. 
    Then $N$ is a zero of $S_n(2\xi)$ if and only if $b_i=0$ for all even~$i$. 
\end{theorem}
\begin{proof}
    By rules B and C, for each $n\in [q',q)$ there is some $k<n$ such that
    $S_n=S_{q'}+S_k$. If $q'$ is odd, then $S_{q'}=1$ and $S_k\geq 0$
    since the walk is one-sided. There are no zeros in $[q',q)$ 
    if $q'$ is odd. If $N$ is a zero, it must
    be in $[q',q)$ for an even $q'$.
    
    We need to prove that the odd digits of $N$ are zero in Ostrowski numeration.
    First suppose that $N<q/2$. Rule C applies and $S_N=S_{q'}+S_k=S_k$ for $N=q'+k$. Therefore, $k$ is a zero and
    by induction has odd digits zero. Note that $k<q/2-q'<q'$ and that $q'$ gives a digit $1$
    at an even position in Ostrowski numeration. The odd digits of $N$ are zero if $N<q/2$.
    If $N\geq q/2$ then $S_N=S_{q'}+S_{k-1}=S_{k-1}$ for $N=q-k$.
    By induction, $N$ is a zero if and only if $N=q-k$ for a zero $k-1<N$. Now
    we use the identity
    \begin{equation}\label{eq:qsom}
        q_{2j+2}=\sum_{i=0}^j a_{2i+2}q_{2i+1}+1    
    \end{equation}
    to write
    \[
    N=\sum_{i=0}^j a_{2i+2}q_{2i+1}-(k-1)
    \]
    for a zero $k-1$. By induction $k-1=\sum_{i=0}^j b_{2i+1}q_{2i+1}$ for digits $b_{2i+1}\leq a_{2i+2}$.
    We conclude that all even digits in the expansion of $N$ are zero.
\end{proof}

It is easy to construct an automaton that decides if all digits on even positions are zero.

\begin{corollary}
    The zeros of $S_n(2\xi)$ form a $\xi$-automatic sequence if $\xi$ is a BR-number.
\end{corollary}

The following generalizes Theorem~\ref{thm:rec2} from $\xi=\sqrt 2$ to BR-numbers. 

\begin{theorem}\label{thm:recordBR}
    Let $\xi\in (0,1)$ be a BR-number and let $N=\sum_{i=0}^n b_iq_i$ be the $\xi$-Ostrowski
    representation. 
    Then $N$ is a record of $S_n(2\xi)$ if and only if $n$ is even, $b_i=0$ for all odd~$i$,
    $b_i={a_{i+1}}/2$ for all even $i<n$, and $b_n\leq a_{n+1}/2$.
\end{theorem}
\begin{proof}
We argue by induction on the index $n$ of the most significant digit.
    Records can only occur in $[q',q)$ if $q'$ is odd, so $n$ is even.
    The first partial quotient $a_1$ is even and the initial record of the walk 
    occurs immediately at the initial steps $a_1/2$, which shows that
    the statement is true for $n=0$.
    Assume that it is true for $n-2$. The final record with most significant
    digit $n-2$ is equal to $(q_{n-1})/2$ by Equation~\eqref{eq:qsom}.
    We write $q_{n+1},q_n,q_{n-1}$ as $q,q',q''$ and $q=aq'+q''$. The final record before $q''$
    is $r=(q''-1)/2$.
    Rule C implies that the records in $[q',q/2)$ are $bq'+r$ up to $q/2$.
    Since $(a/2)q'+r=(q-1)/2$ the digit $b$ runs up to $a/2$.
    Rule B implies that the only possible record in the subinterval $[q/2,q)$
    is $q-k$ if $k-1$ is the final record before $q/2$.
    This record is $(q-1)/2$.
    Since $q$ is odd, the minimal number in $[q/2,q)$ is $(q+1)/2$. If we write it
    as $q-k$, then we get $k=(q-1)/2$ and $k-1$ is below the final record.
    All records are of the prescribed form.
\end{proof}

\begin{corollary}
    Both Conjecture~\ref{conj:1} and the conjecture of Van de Lune and Arias de Reyna hold
    if $\xi$ is a BR-number.
\end{corollary}
\begin{proof}
    We construct an lsd automaton. It checks that $b_i=0$ if $i$ is odd and
    that $b_i\leq a_{i+1}/2$ if $i$ is even. 
    If this inequality is strict,
    then all further digits need to be zero. Therefore, the odd transitions
    are $0$. The even transitions are $a_{i+1}/2$, unless the transition
    is $<a_{i+1}/2$, which moves to a separate state for which all further
    transitions are $0$.
    Since the partial quotients are
    eventually periodic, the automaton eventually loops back and the number
    of states is finite. 
    For a BR number,
    the period is even, so this does not conflict with the check that odd digits
    are zero. 
    This settles Conjecture~\ref{conj:1} for BR-numbers.
    
    To find a system of linear recurrences,
    observe that the difference between consecutive records equals $q_j$ for some even index $j=2k$.
    Each $q_j$ occurs $a_{j+1}/2$ times and therefore the number of occurrences is eventually
    periodic. 
    The convergents $q_{2k}$ form a recurrence sequence and therefore the difference sequence
    of consecutive records satisfies a system of recurrences, and so does the sequence of records. 
\end{proof}

\section{The difference sequence $b-a$}

We establish Kimberling's second observation
that $b-a$ assumes all positive integers infinitely often. 
We focus on the walk \( S_n(2\sqrt{2}) \) and 
for simplicity, we abbreviate our notation to \( S_n = S_n(2\sqrt{2}) \). The walk exhibits two key symmetries, described by Rules B and C.
Rule B is reflexive, meaning that the step from \( S_{q-k} \) to \( S_{q-k+1} \) mirrors the step from \( S_{k-1} \) to \( S_{k-2} \). Rule C, on the other hand, is translational: it states that from the \( q \)-th step onward, the next \( q \) steps replicate the initial \( q \) steps. 
To fully address Kimberling’s second observation, we need an additional symmetry, ensuring that each value \( k > 0 \) appears infinitely often in \( b - a \).

\begin{lemma}\label{lem:ruleD} 
Let $q=q_{2n-1}$ be an even denominator. Then
\[
S_{q/2+k}=S_{q/2}-S_k
\]
for $0\leq k\leq q/2$.
\end{lemma}
\begin{proof}
    This follows from the rules of $S_n(2\sqrt2)$. The intuition behind this equation is that
    after $q/2$ steps the parity of the rounded exponents $2(j+q/2)\sqrt 2$
    is opposite to that of $2j\sqrt 2$ since $q\sqrt 2\approx p$ where $p/q$ is the
    convergent and $p$ is odd.
    
    The half-Pell number $q/2$ is a sum of
    odd convergents $q_{2n-1}/2=q_{2n-2}+q_{2n-4}+\cdots+q_0$. 
    We present $k$ in Pell numeration by a word of length $2n-1$,
    padding with initial zeros if necessary. 
    Since $k\leq q/2$ its Pell presentation is $(10)^j0w$ for some $j\geq 0$
    and a word $w$ of length $2(n-j-1)$. In particular
    \[
    k=q_{2n-2}+\ldots+q_{2n-2j}+k'
    \]
    for $k'<q_{2n-2j-2}$. 
    Rule C implies that $S_k=j+S_{k'}$.   
    If we write $r=q_{2n-2j-1}$ then  $q/2-k=r/2-k'$.
    By rule B and by induction
    \[
    S_{q/2+k}=1+S_{q/2-k-1}=1+S_{r/2-k'-1}=S_{r/2+k'}=S_{r/2}-S_{k'}.
    \]
    Now $q/2$ is the $n$-th record and $r/2$ is the $(n-j)$-th record.
    Therefore $S_{q/2}=j+S_{r/2}$ and we find
    \[
    S_{r/2}-S_{k'}=S_{q/2}-S_k
    \]
    as required.
\end{proof}

\begin{lemma}\label{lem:1}
    Let $q=q_{2n-1}$ be an even denominator. 
    \begin{equation}~\label{eq:ab}
    a(q/2+j)=q+a(j)\text{ and }b(q/2+j)=q+b(j)    
    \end{equation}
     for $j\leq q/2$
\end{lemma}

\begin{proof}
    Since $q$ is even, it is a zero, and there are equally many forward and backward steps to $q$. By definition,
    $a(q/2+j)$ is the index of the $(q/2+j)$-th forward step. There are $q/2$ forward steps up to
    index $q$, after which the walk repeats itself by rule C for the next $q/2$ steps. Therefore, $a(q/2+j)=q+a(j)$. The argument for $b(q/2+j)$ is the same.
\end{proof}

\begin{lemma}
    For any $k>0$, if $b(n)-a(n)=k$ for some $n$, then there are infinitely many $m$
    such that $b(m)-a(m)=k$.
\end{lemma}

\begin{proof}
    It follows from Equation~\eqref{eq:ab} that each difference $b(n)-a(n)$ repeats after $q/2$ steps
    for a sufficiently large $q$.
\end{proof}
If $q=q_{2m-1}$ is the $m$-th even denominator, then $q/2$ is the $m$-th record, and
there is a surplus of $m$ forward steps among the first $q/2$ steps.
\begin{lemma}\label{lem:a(n)}
    For all $m>1$, we have \[
    b\left(\frac {q_{2m-1}+2m}{4}\right)-a\left(\frac {q_{2m-1}+2m}{4}\right)=a(m).
    \]
\end{lemma}
\begin{proof}
    We write $q=q_{2m-1}$.
    There is a surplus of $m$ forward steps; therefore, $(q+2m)/4$ steps are forward and $(q-2m)/4$ are backward.
    The final step of these first $q/2$ steps is forward, since it produces a record.
    By Lemma~\ref{lem:1}, after step $q/2$ the walk repeats itself, but in the opposite direction, for the next $q/2$ steps. Since $a(m)<2m< q/2$ it follows that
    \[
    b\left(\frac {q-2m}{4}+j\right)=\frac q2 + a(j)\text{ for } 1\leq j\leq m.
    \]
    In particular $ b\left(\frac {q+2m}{4}\right)=\frac q2+a(m)$ and $ a\left(\frac {q+2m}{4}\right)=\frac q2$.~\end{proof}

\begin{theorem}\label{thm:kim2}
    For every $k>0$ there are infinitely many $j$ such that $b(j)-a(j)=k$.
\end{theorem}

\begin{proof}
    The previous lemmas take care of the case that $k$ is in the sequence $a$. 
    We need to find a solution for $k$ in $b$, say $k=b(n)$.
    Take a sequence of $n$ odd denominators of convergents $c_n> c_{n-1}> \ldots\> c_{1}$ such that
    $c_1/2>b(n)$.
    These denominators do not need to be consecutive, which is why we write $c$ instead of $q$, to avoid confusion. 
    Let $d=c_n+c_{n-1}+\cdots c_{1}$. Rule C implies that $S_d=n$ and that the $d$-th step is forward.
    There is a surplus of $n$ forward steps which implies that $a\left(\frac{n+d}2\right)=d$. 
    By Rule C the deficit of $n$ backward steps is compensated for at index $d+b(n)$. We conclude that
    $b\left(\frac{n+d}2\right)-a\left(\frac{n+d}2\right)=b(n)$.
\end{proof}

\section{Self-similarity of noble mean rotations}

A quadratic number $\xi^2=m\xi+1$ with constant continued fraction expansion $\xi=[m;m,m,m,\ldots]$
is called a \emph{metallic mean} or \emph{noble mean}~\cite{baakegrimm}. 
For $m=1$ it is the golden mean and for $m=2$ it is the silver mean.
A noble mean is a BR-number if $m$ is even. Our standing example $S_n(2\sqrt 2)$
is generated by the rotation over the silver mean.
It is well-known that the irrational rotation $\rho\colon x\mapsto x+\xi\mod 1$
has a self-similarity for noble means.
We shall see that this explains why the deterministic random walk $S_n(2\xi)$ is non-negative,
for noble means with even $m=2k$. 
Only the fractional part of $\xi$ matters for the
rotation, which is why we subtract $m$ from the noble mean, adjusting it to $\xi_m=[0;m,m,m,\ldots]$.

The rotation $\rho$ is an \emph{interval exchange transformation}~\cite{keane1975}, which adds $\xi$ to $x\in [0,1-\xi)$
and subtracts $1-\xi$ from $x\in [1-\xi,1).$ 
Recall that the return of $x$ to $I\subset [0,1)$  is the iterated image
$\rho^n(x)\in I$ such that no $\rho^j(x)$ for $0<j<n$ is in $I$. 

\begin{lemma}\label{lem:selfsimilar}
    Let $\xi_m=\frac{\sqrt{m^2+4}-m}2$ be the adjusted noble mean.
    The return map to $[0,1-m\xi)$ is a rescaling of
    the original rotation on $[0,1)$.
\end{lemma}
\begin{proof}
The first two convergents of $\xi_m$ are 
\[
\frac{m}{m^2+1}<\xi_m<\frac 1m
\]
and the closest returns to $0$ among its first $m^2+1$ rotations
are $m\xi_m-1<0<(m^2+1)\xi_m-m$.
We suppress the index and write $\xi$ instead of $\xi_m$.
The first rotation $\{j\xi\}$ that ends up in $[0,1-m\xi)$ is $(m^2+1)\xi-m$.
Let $R$ be the return map to  $[0,1-m\xi)$. 
A straightforward computation
gives
\[
R(x)=\left\{\begin{array}{ll} x+(m^2+1)\xi-m&\text{if }x<(m+1)-(m^2+m+1)\xi, \\
                              x+(m^2+m+1)\xi-m-1&\text{if }x\geq (m+1)-(m^2+m+1)\xi. 
                              \end{array}\right.
\]
If we rescale this interval exchange to unit length, we get a rotation over
\[
\frac{(m^2+1)\xi-m}{1-m\xi}=\xi.
\]
\end{proof}

We partition the circle into three labeled intervals \[\{a,b,c\} = \{ [0,1/2), [1/2,1-\xi), [1-\xi,1)\}.\]
The itinerary of $x\in [0,1/2)$ up to its return $S(x)$ is $a$ for $x\in I_1$ because the return is immediate.
It is $ab^{k+1}c$ for $x\in I_2$ and $ab^kc$ for $x\in I_3$.

%
%
%

\begin{theorem}\label{thm:returnmap}
Let $\xi_m$ be the fractional part of the noble mean for an even $m=2k$.
The sequence $(a_n)_{n \geq 0}$ for $a_n = 1_{[0,1/2)} ( \{ n \xi_m \} )$ equals $\lim_{k \to \infty} \tau \circ \sigma^k(a)$
 for the substitution and coding  
 $$
 \sigma: \begin{cases}
                                  a \to a^{k+1}b^{k-1}c \, (a^kb^{k-1}c)^{p-1},\\
                                  b \to a^kb^kc \, (a^kb^{k-1}c)^{p-1},\\
                                  c \to a^kb^kc \, (a^{k}b^{k-1}c)^p,
                                 \end{cases}
\qquad \text{ and } \qquad
\tau: \begin{cases}
    a \to 1,\\
    b \to 0,\\
    c \to 0.
\end{cases}
$$
\end{theorem}

\begin{proof}
The return map to $[0,1-m\xi)$ is a rescaling of the rotation, in which the partition
$\{a,b,c\}$ is scaled to 
\[\{a',b',c'\} = \{ [0,1/2-k\xi), [1/2-k\xi,(1-\xi)(1-m\xi)), [(1-\xi)(1-m\xi),1-m\xi)\}.\]
The return time is $m^2+1$ on $a'\cup b'$ and $m^2+m+1$ on $c'$. The substitution $\sigma$
gives the itinerary of elements of $\{a',b',c'\}$ through $\{a,b,c\}$ until their return. 
Note that the return times correspond to the lengths of the substitution. 

After $m$ rotations, $[0,1-m\xi)$ is in $[m\xi,1)$ and after one more rotation it 
is in $[\xi,1-(m-1)\xi)$. For an element of $a'$, the itinerary if $a^{k+1}b^{k-1}c$ after
$m$ rotations, including its initial position in $a$. For the other elements, the
itinerary is $a^kb^kc.$ 
That explains the prefixes of the coding $\sigma$.
After $k$ rotations, the initial point of $b'$ reaches $\frac 12$,
which marks the transition from coding 1 to coding 0. This is where we need that $m$ is even.

For each $x\in [\xi,1-(m-1)\xi)$, the itinerary of its initial position and a subsequent
$m-1$ rotations is $a^kb^{k-1}c$. This itinerary of blocks of $m$ remains the same until
the return to $[0,1-m\xi)$. It follows that $\sigma$ describes the return map, if we 
associate $\{a',b',c'\}$ to $\{a,b,c\}$. The coding $\tau$ corresponds to the indicator function.
\end{proof}

As a consequence, we get another proof that deterministic random walks are one-sided for
noble means that are BR-numbers.

\begin{corollary}
    The deterministic random walk $S_n(2\xi)$ is nonnegativefor noble means with even $m$.
\end{corollary}
\begin{proof}
    If we put $+1$ for $a$ and $-1$ for $b$ and $c$ then the running sum of \[a^{k+1}b^{k-1}c \, (a^kb^{k-1}c)^{m-1},\]
    which is the substitution word of $a$, is positive. The running sum of the other two substitution words,
    \[
    a^kb^kc \, (a^kb^{k-1}c)^{m-1}
    \qquad \text{ and } \qquad 
    a^kb^kc \, (a^{k}b^{k-1}c)^m,
    \] 
    is $\geq -1$. The deterministic random walk
    $S_n(\xi)$ is a running sum of substitution words. If it were negative at some point, then it must have a surplus
    of $b$ and $c$ over $a$ at an earlier index, which is nonsense.
\end{proof}

Lemma \ref{lem:selfsimilar} holds for all noble means and Theorem~\ref{thm:returnmap} 
can be extended to all noble means as well, but it gets more elaborate. For a noble
mean $\xi_m$ with even $m$, the mid-point of the interval $[0,1-m\xi)$ is in the 
backward orbit of $\frac 12$ and that is why the partition $\{a,b,c\}$ carries over
under rescaling. If $m$ is odd, we need to take a smaller interval. 

If $p_n/q_n<\xi$ is a convergent, then the return map
to $[0,p_n-q_n\xi)$ is a rescaling of the rotation. 
We need
$q_n$ to be even to preserve the partition, and we need
$n$ to be odd so that $p_n-q_n\xi>0$.
For example, for the golden mean $\xi_1 = \frac12 (\sqrt{5}-1)$,
we can choose the interval $[0, 5-8\xi_1)$.
The corresponding substitution and coding are
$$
\sigma : \begin{cases}
 a \to acacbacaccacb\\
 b \to acacbacaccacbacaccacb\\
 c \to acaccacaccacbacaccacb
\end{cases}
\quad \text{ and } \quad
\tau : \begin{cases}
               a \to 1\\
               b \to 1\\
               c \to 0
              \end{cases}
$$
The running sums of the substitution words for $a$ and $b$ are
non-negative. The minimum of the running sum of the word for $c$
is $-2$. This indicates that, as we know from Theorem~\ref{thm:BR}, 
the walk is unbounded in both directions.

\section{Concluding questions and remarks}

We established a connection between deterministic random walks and automata in certain cases, but much remains to be explored. The general deterministic random walk is defined by
\begin{equation}\label{eq:walk}
\sum_{j=1}^n (-1)^{\lfloor j\xi+\gamma\rfloor }.
\end{equation}
We only considered the homogeneous case with offset \(\gamma = 0\) and \(\xi\) as a BR-number. Do our results extend to general walks and arbitrary quadratic irrationals? Van de Lune and Arias de Reyna found that the records of general walks for quadratic \(\xi\) appear to satisfy a recurrence for certain \(\gamma\). Is it true that the records of this general walk form a \(\xi\)-automatic set if \(\gamma\) is rational? Interestingly, the set of zeros can be empty for such random walks. Ralston~\cite{ralston20141} proved that for every $\xi$ there exists a unique
$\gamma$ such that the random walk in Equation~\ref{eq:walk} is positive.

\section{Acknowledgement}
We would like to thank Jeffrey Shallit for helpful comments and for drawing our attention to Problem B6 of the 81st Putnam competition.




\bibliographystyle{siam}
\bibliography{a120243}

\begin{thebibliography}{10}

\bibitem{aaronson1982visitors}
{\sc J.~Aaronson and M.~Keane}, {\em The visits to zero of some deterministic random walks}, Proc. London Math. Soc., 3 (1982), pp.~535--553.

\bibitem{avila2015visits}
{\sc A.~Avila, D.~Dolgopyat, E.~Duryev, and O.~Sarig}, {\em The visits to zero of a random walk driven by an irrational rotation}, Isr. J. Math., 207 (2015), pp.~653--717.

\bibitem{baakegrimm}
{\sc M.~Baake and U.~Grimm}, {\em Aperiodic order}, vol.~1, Cambridge University Press, 2013.

\bibitem{beck2010}
{\sc J.~Beck}, {\em Randomness of the square root of 2 and the giant leap, part 1}, Period. Math. Hung., 60 (2010), pp.~137--242.

\bibitem{bhargava}
{\sc M.~Bhargava, K.~Kedlaya, and L.~Ng}, {\em Solutions to the 81st {W}illiam {L}owell {P}utnam mathematical competition},  (2021).
\newblock Online \url{https://kskedlaya.org/putnam-archive/2020s.pdf}.

\bibitem{boshernitzan2009}
{\sc M.~Boshernitzan and D.~Ralston}, {\em Continued fractions and heavy sequences}, Proc. Am. Math. Soc., 137 (2009), pp.~3177--3185.

\bibitem{bruin2024}
{\sc H.~Bruin, C.~Fougeron, D.~Ravotti, and D.~Terhesiu}, {\em On asymptotic expansions of ergodic integrals for $\mathbb {Z}^{d}$-extensions of translation flows}, arXiv:2402.02266,  (2024).

\bibitem{drmota2006}
{\sc M.~Drmota and R.~F. Tichy}, {\em Sequences, discrepancies and applications}, Springer, 2006.

\bibitem{dupain1980}
{\sc Y.~Dupain and V.~T~S{\'o}s}, {\em On the one-sided boundedness of the discrepancy-function of the sequence $\{$n$\alpha$$\}$}, Acta Arith., 37 (1980), pp.~363--374.

\bibitem{fokkink1994}
{\sc R.~Fokkink, W.~Fokkink, and J.~van~de Lune}, {\em Fast computation of an alternating sum}, Nieuw Arch. Wisk., 12 (1994), pp.~13--18.

\bibitem{keane1975}
{\sc M.~Keane}, {\em Interval exchange transformations}, Math. Z., 141 (1975), pp.~25--31.

\bibitem{kesten1966}
{\sc H.~Kesten}, {\em On a conjecture of {E}rd{\"o}s and {S}z{\"u}sz related to uniform distribution mod 1}, Acta Arith., 12 (1966), pp.~193--212.

\bibitem{lune2008}
{\sc J.~v.~d. Lune and J.~Arias~de Reyna}, {\em On some oscillating sums}, Unif. Distrib. Theory, 3 (2008), pp.~35--72.

\bibitem{mousavi2016}
{\sc H.~Mousavi}, {\em Automatic theorem proving in {W}alnut},  (2016).
\newblock Online \url{https://arxiv.org/abs/1603.06017}.

\bibitem{obryant2006}
{\sc K.~O'Bryant, B.~Reznick, and M.~Serbinowska}, {\em Almost alternating sums}, Am. Math. Mon., 113 (2006), pp.~673--688.

\bibitem{ralston20141}
{\sc D.~Ralston}, {\em 1/2-heavy sequences driven by rotation}, Monatsh. Math., 175 (2014), pp.~595--612.

\bibitem{schaeffer2024}
{\sc L.~Schaeffer, J.~Shallit, and S.~Zorcic}, {\em Beatty sequences for a quadratic irrational: Decidability and applications}, arXiv:2402.08331,  (2024).

\bibitem{shallit2022}
{\sc J.~Shallit}, {\em The logical approach to automatic sequences: Exploring combinatorics on words with Walnut}, vol.~482, Cambridge University Press, 2022.

\bibitem{sos1957}
{\sc V.~T. S{\'o}s}, {\em On the theory of diophantine approximations. {I}}, Acta Math. Acad. Sci. Hungar., 8 (1957), pp.~461--472.

\bibitem{lune1984}
{\sc J.~van~de Lune}, {\em Sums of equal powers of positive integers}, PhD thesis, CWI Amsterdam, 1984.

\end{thebibliography}

\end{document}